\documentclass{amsart}

\usepackage[latin1]{inputenc}
\usepackage{amsfonts}
\usepackage{amsmath}
\usepackage{amsthm}
\usepackage{amssymb}
\usepackage{latexsym}
\usepackage{enumerate}

\newtheorem{theorem}{Theorem}[section]
\newtheorem{lemma}[theorem]{Lemma}
\newtheorem{proposition}[theorem]{Proposition}

\newtheorem{definition}[theorem]{Definition}

\numberwithin{equation}{section}

\begin{document}

\newcommand{\cc}{\mathfrak{c}}
\newcommand{\N}{\mathbb{N}}
\newcommand{\Q}{\mathbb{Q}}
\newcommand{\R}{\mathbb{R}}

\newcommand{\PP}{\mathbb{P}}
\newcommand{\forces}{\Vdash}
\newcommand{\dom}{\text{dom}}
\newcommand{\osc}{\text{osc}}

\title{On biorthogonal systems whose functionals are finitely supported}

\author{CHRISTINA BRECH}
\thanks{The first author was supported by FAPESP fellowship (2007/08213-2), 
which is part of Thematic Project FAPESP (2006/02378-7). Part of the research was done
at the Technical University of \L\'od\'z where the first author was partially supported by Polish Ministry of
Science and Higher Education research grant N N201
386234.} 
\address{Instituto de Matem\'atica, Estat\'\i stica e Computa\c c\~ao Cient\'\i fica, Universidade Estadual de Campinas, 
Rua S\'ergio Buarque de Holanda, 651 - 13083-859, Campinas, Brazil}
\curraddr{Departamento de Matem\'atica, Instituto de Matem\'atica e Estat\'\i stica, Universidade de S\~ao Paulo,
Rua do Mat\~ao, 1010 - 05508-090, S\~ao Paulo, Brazil}
\email{christina.brech@gmail.com}

\author{PIOTR KOSZMIDER}
\thanks{The second  author was partially supported by Polish Ministry of
Science and Higher Education research grant N N201 386234. Part of the research was done
at the State  University of Campinas UNICAMP where the second
 author was partially supported by the Department of Mathematics. }

\address{Instytut Matematyki Politechniki \L\'odzkiej,
ul.\ W\'olcza\'nska 215, 90-924 \L\'od\'z, Poland}

\email{\texttt{pkoszmider.politechnika@gmail.com}}

\date{}

\begin{abstract} We show that for each natural $n>1$  it is consistent that
there is a compact Hausdorff space $K_{2n}$
such that in $C(K_{2n})$ there is no uncountable (semi)biorthogonal 
sequence $(f_\xi,\mu_\xi)_{\xi\in \omega_1}$ where $\mu_\xi$'s are atomic measures
with supports consisting of at most $2n-1$ points of $K_{2n}$, but there are
biorthogonal systems $(f_\xi,\mu_\xi)_{\xi\in \omega_1}$ where $\mu_\xi$'s are atomic measures
with supports consisting of $2n$ points. 
This complements a result of Todorcevic that it is consistent that each
nonseparable Banach space  $C(K)$  has an uncountable biorthogonal system where
the functionals are measures of the form $\delta_{x_\xi}-\delta_{y_\xi}$
for $\xi<\omega_1$ and $x_\xi,y_\xi\in K$.
It also follows that
  it is consistent that the irredundance of
the Boolean algebra $Clop(K)$ or the Banach algebra $C(K)$
for  $K$  totally disconnected can be strictly smaller
than the sizes of biorthogonal systems in $C(K)$.
The compact spaces exhibit an interesting behaviour with respect to
known cardinal functions:  the hereditary density of the powers $K_{2n}^k$
is countable up to $k=n$ and it is uncountable (even the spread
is uncountable) for $k>n$.

\end{abstract}

\maketitle

\section{Introduction}

If $X$ is a Banach space and $X^*$ its dual,
then $(x_i,x_i^*)_{i\in I}\subseteq X\times X^*$ is called a biorthogonal
system if $x_i^*(x_i)=1$ and $x^*_i(x_j)=0$ if $i\not=j$
for each $i,j\in I$. If $\alpha$ is
an ordinal, a transfinite
sequence $(x_i,x_i^*)_{i< \alpha}\subseteq X\times X^*$ is called a semibiorthogonal
sequence if $x_i^*(x_i)=1$, $x^*_i(x_j)=0$ if $j<i<\alpha$
and  $x^*_i(x_j)\geq 0$ if $i<j<\alpha$. 

Biorthogonal systems have always played an important role in the theory
of Banach spaces (\cite{biorthogonal}) because all kinds of bases in
Banach spaces are in particular $X$-parts of biorthogonal 
systems (\cite{singer1} and \cite{singer2}).
Semibiorthogonal sequences have been introduced quite recently
(\cite{borwein}) in the relation with the  sets
in Banach spaces  supported by all of their points (\cite{rolewicz},
\cite{lazar}, \cite{convex}).

We will mainly deal with biorthogonal systems in Banach spaces $C(K)$
of all continuous functions on a compact Hausdorff space $K$ with the
supremum norm. Its dual space is isometric to the Banach space $M(K)$
of all Radon measures on $K$ with the variation norm, and so, we will identify
this dual with $M(K)$.
If $K$ is a compact Hausdorff space and $x\in K$, 
$\delta_{x}$ denotes the functional on $C(K)$ defined by $\delta_x(f)=f(x)$
for all $f\in C(K)$.

This paper is motivated by the following question: {\it If 
there is an uncountable biorthogonal system
 $(x_\xi,x^*_\xi)_{\xi\in \omega_1}$ in $C(K)\times M(K)$,  is there also
one such that $$x^*_\xi=\delta_{x_\xi}-\delta_{y_\xi}$$
for some points $x_\xi,y_\xi\in K$?}
We will follow \cite{mirnaistvan} and call such a biorthogonal
system a nice biorthogonal system.

The origin of this question is that in all concrete situations so far analyzed
in the literature the above question
has positive answer. Moreover, it happens for a good reason, namely,
it follows from a recent result
of Todorcevic that Martin's axiom together with the negation of
the continuum hypothesis implies the positive answer to this question.
Indeed, analyzing the proof of Theorem 11 of \cite{stevobio},
one gets two cases: the first when $K$ is hereditarily separable, which is the
main part of that proof and the constructed biorthogonal system is nice;
and the second case, when $K$ is c.c.c. but contains a nonseparable
subspace, then the proof of Theorem 10 of \cite{stevobio} provides
the required nice system; if $K$ is not c.c.c., one can easily
obtain an uncountable nice biorthogonal system.

There is one more reason why nice biorthogonal systems
appear frequently in the context of Banach spaces $C(K)$ 
and which makes them more meaningful.
Namely, $(f_\alpha)_{\alpha\in \kappa}$ is the $X$-part
of a nice biorthogonal system if and only if
$(f_\alpha)_{\alpha\in \kappa}$ is irredundant 
in the Banach algebra $C(K)$, in the sense that
no $f_\alpha$  belongs to the Banach subalgebra generated by the
remaining elements. This is a consequence
of the Stone-Weierstrass theorem.
If $K$ is totally disconnected and  $f_\alpha$'s are characteristic
functions of clopen $A_\alpha\subseteq K$, we obtain
a well-known notion of an irredundant set in a Boolean algebra (see e.g.,
\cite{monk}),
i.e., a set where no element belongs to the Boolean  algebra
generated by the remaining elements.
The irredundance of a Boolean algebra is the supremum of
cardinalities of irredundant sets. 

To formulate properly our main results
we need the following:

\begin{definition} Let $K$ be a compact Hausdorff space
and $n\in \N$. We say that the functionals of 
a sequence $(f_\xi,\mu_\xi)_{\xi\in \omega_1}\subseteq C(K)\times M(K)$
are $n$-supported if each $\mu_\xi$ is an atomic measure
whose support consists of no more than $n$ points of $K$.
\end{definition}

\begin{theorem} For each natural  $n>1$, it is consistent that
there is a compact Hausdorff space $K_{2n}$ such that in 
$C(K_{2n})$ there is no uncountable semibiorthogonal 
sequence whose functionals are $2n-1$-supported, but there are
biorthogonal systems whose functionals are $2n$-supported.  

Moreover, $K_{2n}^n$ is hereditarily separable
but $K_{2n}^{n+1}$ has an uncountable discrete subspace. 
Neither the Banach algebra $C(K_{2n})$ nor the
Boolean algebra  $Clop(K_{2n})$ have
an uncountable irredundant family. In particular, $C(K_{4})$
has an uncountable biorthogonal system but it has no uncountable nice biorthogonal system.
\end{theorem}

Such a situation suggests many questions about the size
of biorthogonal systems of various types in $C(K)$ spaces as well as
in general Banach spaces. These more general discussions will
appear elsewhere. In particular, we are unable to obtain $K$'s such that
$C(K)$ contains biorthogonal systems whose functionals are $2n+1$-supported but does not
contain one whose functionals are $2n$-supported. 
The reason why some fundamental change with the approach
must be taken to obtain  such a space is shown in Lemma \ref{lemmathree}.

On the other hand, if $n=1$ one has absolute results. 
If $K$ is the split interval, then $K$ is hereditarily separable and so cannot
have an uncountable semibiorthogonal system whose functionals are $1$-supported, 
but $C(K)$ has an uncountable biorthogonal system
(see \cite{godefroy}).

It seems that our compact space is the first example showing that the
hereditary density or spread of finite powers of a compact space may
change its value from countable to uncountable arbitrarily high in $\N$.
Such an example can be only consistent since, for example, under MA$+\neg$CH
if $K^3$ is  hereditarily separable for a compact $K$, then it is metrizable,
and so, all finite powers are hereditarily separable. This
follows from the fact that then there are no compact $S$-spaces  (\cite{sspaces}), from
the Katetov theorem (\cite{katetov}) and from the fact that Lindel\"of regular spaces are normal.

The paper is organized as follows: in the following second
section we discuss
a general form of the compact spaces we construct and call
them unordered $N$-split Cantor sets. They are versions
of the split interval whose relation with biorthogonal systems 
in Banach spaces was already demonstrated in \cite{godefroy}.
Section 3 is devoted to a generic construction of  Boolean algebras
whose Stone spaces are the $K_{2n}$'s. This is the only section that requires
the knowledge of forcing.
The partial order used is a new modification
of that of \cite{piotrrolewicz}, which produced nonseparable
$C(K)$'s with no uncountable semibiorthogonal sequences. Thus
our spaces are quite controllable members of  the  group of compact spaces
constructed in \cite{bellginsburgstevo}, \cite{shelah1}, \cite{shelah2},
\cite{piotrrolewicz}. In this section we also prove the existence
of an uncountable discrete subspace of $K_{2n}^{n+1}$ and
an uncountable biorthogonal system in $C(K)$ whose functionals are $2n$-supported.
The section ends with Theorem \ref{generictheorem}, which
expresses the random character of the constructed compact space.
Later on we use this theorem to prove further properties of the
space. Hence a reader not familiar with forcing may use this theorem
for other purposes and read only the following section.

The last, fourth section is devoted to applications of Theorem 
\ref{generictheorem}, that is the proof that  $K_{2n}^n$ is
hereditarily separable and that $C(K_{2n})$ has no uncountable
semibiorthogonal sequences whose functionals are $2n-1$-supported.

We use standard notation. In particular
$[n]=\{1, \dots, n\}$ and $n=\{0,..., n-1\}$ for a positive natural number $n$. 
$A^B$ denotes the set of all functions from $B$ into $A$ and so
if $2=\{0,1\}$ we have that $2^\omega$ denotes all infinite sequences
with terms in $\{0,1\}$, while $2^n$ stands for functions from $n$ into $\{0,1\}$;
also $2^{<\omega}=\bigcup\{2^n:n\in \N\}$.
$\langle s\rangle=\{x\in 2^\omega:\ s\subseteq x\}$ for $s\in 2^n$ for some $n\in \N$.
If $A, B$ are sets of ordinals, then
$A<B$ means that $\alpha<\beta$ for any $\alpha\in A$ and any $\beta\in B$.

\section{Unordered $N$-split Cantor sets}

Fix a sequence of distinct elements ${\mathcal X}=\{x_\xi:\xi<\omega_1\}\subseteq 2^\omega$ and $N\in\N$.
Let
$$K_N=(2^\omega\setminus\mathcal X)
\cup ({\mathcal X}\times [N])$$
and define
$$V_s=(\langle s\rangle \cap(2^\omega\setminus{\mathcal X}))\cup( (\langle s\rangle \cap{\mathcal X})
\times [N]).$$

\begin{definition}\label{splittingFamily}
A family $(A_{\xi,i}: \xi<\omega_1, i\in [N])$ of subsets
of $K_N$ is called an $N$-splitting family if it satisfies the following conditions:
\begin{enumerate}
\item $(x_\xi,i)\in A_{\xi,i}\subseteq K_N$ for each $\xi<\omega_1$ and $i\in [N]$;
\item  for each $\xi<\omega_1$ the sets $A_{\xi,i}$'s are pairwise disjoint;
\item  for each $\xi<\omega_1$ we have $K_N=A_{\xi,1}\cup...\cup A_{\xi,N}$;
\item if  $\eta<\xi$, then there is $k\in \N$ 
and $ j\in[N]$ such that
$A_{\eta,i}\cap V_{x_\eta|k}\subseteq A_{\xi,j}\cap V_{x_\eta|k}$;
\item if $\eta > \xi$ and $x=x_\eta$
or $x\in  2^\omega\setminus{\mathcal X}$, then there is $k\in \N$ and $j\in[N]$  such that
$V_{x|k}\subseteq A_{\xi,j}$.
\end{enumerate}
\end{definition}

\begin{definition} 
Given an $N$-splitting family $(A_{\xi,i}: \xi<\omega_1, i\in [N])$, 
we call the space $(K_N,{\mathcal T})$ an unordered $N$-split Cantor
set if the topology
$\mathcal T$  on  $K_N$ is defined  by 
indicating neighbourhood bases ${\mathcal B}_x$ at
$x$ for every $x\in K_N$ in the following way:  if
$x\in 2^\omega\setminus{\mathcal X},$ 
then  $${\mathcal B}_x=\{V_s: s\subseteq x\}$$ and if 
 $x=(x_\xi, j)\in K_N$, then
$${\mathcal B}_x=\{V_s\cap A_{\xi,j}: s\subseteq x_\xi\}.$$
\end{definition}

The intuitive meaning of the above definitions
is as follows: each point $x_\xi$ of $2^\omega$ is split into $N$ points
$(x_\xi, 1),..., (x_\xi, N)$. If we view $K_N$ as constructed inductively,
when at step $\xi<\omega_1$ we construct the splitting clopen neighbourhoods
$A_{\xi,1},..., A_{\xi, N}$ of the
points $(x_\xi, 1),..., (x_\xi, N)$, then these neighbourhoods split only
$x_\xi$ and no other  previously constructed $(x_\eta, i)$ for $\eta<\xi$
(condition \ref{splittingFamily}.(4)) nor $x_\eta$ for $\eta>\xi$ nor $x\in 2^\omega\setminus{\mathcal X}$ (condition \ref{splittingFamily}.(5)).
Note that, on the other hand, $A_{\xi, i}$'s may split $x_\eta$ for $\eta<\xi$, and in this case, by
condition \ref{splittingFamily}.(4), they do it ``the same way" as the $A_{\eta,j}$'s.

\begin{proposition}\label{splitCantor}
Let $N\in \N$. If $(A_{\xi,i}: \xi<\omega_1, i\in [N])$ is an $N$-splitting family,
then the corresponding unordered $N$-split Cantor set is a compact, Hausdorff,
totally disconnected topological space.
\end{proposition}
\begin{proof}
Since $V_\emptyset=K_N$, conditions (1) - (3) of Definition \ref{splittingFamily} imply that $A_{\xi,i}$'s are clopen sets.
Now using Proposition 1.2.3. of \cite{engelking}, we will prove that the above families 
satisfy the axioms for neighbourhood bases BP1-BP3 from \cite{engelking}. The only nontrivial
part is to prove
that given $x\in V\in {\mathcal B}_y$, there is
$U \in {\mathcal B}_x$ such that $x\in U\subseteq V$.

Suppose $x\in 2^\omega\setminus  {\mathcal X}$
and $x\in V_s\in {\mathcal B}_y$. Then $s\subseteq x$ and so $V_s$ itself  is
in ${\mathcal B}_x$. If 
$x\in V_s\cap A_{\xi,i}$, we also have $s\subseteq x$ and
by (5) of Definition \ref{splittingFamily} there is $k\in \N$ such that $V_{x|k}\subseteq A_{\xi,j}$
for some $j\in N$. Put $t=s\cup x|k$ and note that we have that 
$V_{t}\subseteq A_{\xi,j}$, so by the disjointness (condition \ref{splittingFamily}.(2))
we 
have $j=i$ with $x\in V_t\in {\mathcal B}_x$ and
$V_t\subseteq V_s\cap A_{\xi,i}$.

Now suppose that $x=(x_\eta,i)$ and $x\in V_s\in {\mathcal B}_y$, hence
$s\subseteq x$ and so $V_s\cap A_{\eta,i}\in {\mathcal B}_x$ and
$x\in V_s\cap A_{\eta,i} \subseteq V_s$.

 Finally, let
$x=(x_\eta,i)$ and $x\in V_s\cap A_{\xi,j}\in {\mathcal B}_{(x_\xi,j)}$, then $s\subseteq x_\eta$.

First consider $\eta<\xi$, then by (5) of Definition \ref{splittingFamily}
there are $k\in \N$ and $j'$ such that
$A_{\eta,i}\cap V_{x_\eta|k}\subseteq A_{\xi,j'}\cap V_{x_\eta|k}$
and by disjointness (2) we get that $j'= j$. So, if we put
$t=s\cup x_\eta|k$, then
$A_{\eta,i}\cap V_t\subseteq A_{\xi,j}\cap V_t\subseteq
 A_{\xi,j}\cap V_{s}$
and of course $A_{\eta,i}\cap V_t\in {\mathcal B}_{(x_\eta,i)}$.

Secondly if  $\eta\geq \xi$ and
$(x_\eta,i)\in V_s\cap A_{\xi,j}$, we also have $s\subseteq x_\eta$ and
by \ref{splittingFamily}.(4) there are $k\in \N$ and $j'$ such that $V_{x_\eta|k}\subseteq A_{\xi,j'}$
for some $j'$. By the disjointness we 
have $j=j'$. 
If $t=s\cup x_\eta|k$ we have that 
$V_{t}\subseteq A_{\xi,j}$, so $x\in V_t\in {\mathcal B}_x$ and
$V_t\subseteq V_s\cap A_{\xi,i}$. This completes the proof that
${\mathcal B}_x$'s form a local neighbourhood system.

The Hausdorff property is easy since basic sets are clopen.

To prove the compactness, suppose $\mathcal U$ is an open cover of $K_N$.
We may assume that it consists of basic open sets.
For each  $x\in  2^\omega\setminus{\mathcal X}$
define $s_x\in 2^{<\omega}$ such that $x\in V_s\subseteq U\in \mathcal U$
for some $U$, and for each $\xi<\omega_1$ define $s_\xi\in 2^{<\omega}$ such that $(x_\xi,i)\in V_{s_\xi}\cap A_{\xi,i}\subseteq U\in \mathcal U$
for some $U$, and for each $1\leq i\leq N$. This actually gives by (3) of Definition \ref{splittingFamily} that
$V_{s_\xi}$ is covered by finitely many $U\in \mathcal U$.

Now $\{\langle s_x\rangle,\langle s_\xi\rangle: x\in 2^\omega\setminus{\mathcal X},
\xi<\omega_1\}$ forms an open cover of $2^\omega$ which is compact and so
it has a finite subcover, which easily yields a finite subcover of $\mathcal U$.
\end{proof}

\begin{definition}\label{cantornotation} Suppose $N\in \N$ and $K_N$ is an unordered $N$-split Cantor set.
Under the notation as above, we 
define the following:
\begin{itemize}
\item
$R_\xi=\{(x_\xi,1),...,(x_\xi,{N})\}$,
\item  ${\mathcal A}_\alpha$ is  the subalgebra of $Clop(K_N)$
generated by $(V_s: s\in 2^{<\omega})$ and $\{A_{\xi,i}:\xi<\alpha, i \in [N]\}$
for $\alpha\leq\omega_1$.
\item $C_\alpha$
is the closure (in the norm) of
finite linear combinations of characteristic functions
of elements of ${\mathcal A}_\alpha$ inside $C(K)$.

\end{itemize}
\end{definition}

Note that $C_0$ can be naturally identified with
$C(2^\omega)$ inside $C(K)$.

\begin{lemma}\label{lemma3} Let $N\in\N$ and let $K_N$ be an unordered
$N$-split Cantor set.
For every $n\in\N$ and for every $\alpha\in\omega_1$ and every $i\in [N]$
we have
$$A_{\alpha,i} \setminus  V_{x_\alpha|n}\in {\mathcal A}_\alpha.$$
\end{lemma}

\begin{proof}  By the properties \ref{splittingFamily}.(4) and (5) of $A_{\xi,i}$'s any point of 
$K_N \setminus R_{\alpha}$ has a neighbourhood
$V$ such that for every $i\in [N]$
it is included in $A_{\alpha,i}$ or disjoint from $A_{\alpha,i}$ 
and moreover $V\in {\mathcal A}_\alpha$.

Since $A_{\alpha,i} \setminus V_{x_\alpha|n}$ is a compact subspace
of $K_N\setminus R_\alpha$, we have a finite subcover
consisting of subsets i.e., $A_{\alpha,i}\setminus V_{x_\alpha|n}$ 
is the supremum of a finite family
of elements of ${\mathcal A}_\alpha$ as required.
\end{proof}

Let us see the general form
of continuous rational simple functions on an unordered $N$-split Cantor set.
By a rational simple function we mean a function
assuming only finitely many rational values.

\begin{lemma}\label{lemmaformoffunction}Suppose that 
$N\in\N$ and that  $K_N$ is an unordered $N$-split
Cantor set, $\varepsilon>0$, $\mu$ is a (regular) Radon measure on $K_N$ and that
$f$ is a continuous rational simple function
on  $K_N$. Then there is
a simple rational function $g\in C(2^\omega)$,
distinct $\xi_1,..., \xi_k<\omega_1$ and
there are rationals $q_{i,l}$, non-negative integers $m_i$
and $s_i\in 2^{m_i}$ such that  $s_i=x_{\xi_i}|m_i$      for $1\leq i\leq k\in
\omega$ and for $1\leq l< N$ such that
$$f=g+\sum_{1\leq i\leq k}\sum_{1\leq l< N}q_{i,l}\chi_{ A_{\xi_{i,l}}\cap V_{s_i}}$$
and such that
$$\sum_{1\leq i\leq k}\max_{1\leq l<N}(|q_{i,l}|)|\mu|(V_{s_i}\setminus R_{\xi_i})\leq \varepsilon.$$
\end{lemma}

\begin{proof}  
By induction on $\xi$ we prove that
any continuous simple rational function
in $C_\xi$ can be written in the form as in the lemma. The Stone-Weierstrass theorem and
the uncountable cofinality of $\omega_1$ imply that the union of $C_\xi$'s is
the entire $C(K_N)$.

The limit stage is trivial. So, suppose we are
done for $C_\xi$ and we are given a continuous 
simple rational function $f$ in  $C_{\xi+1}$.
Note that
$$\bigcap_{m\in \N} V_{x_\xi|m}=R_{\xi}.$$
Hence, 
the regularity of the Radon measures implies that 
$|\mu|(V_{x_{\xi}|m} \setminus R_{\xi})$'s converge to $0$.
Let $m_1$ be such that 
$$|\mu|(V_{x_{\xi}|m} \setminus R_{\xi})\leq{\varepsilon\over{4||f||}}$$
for $m\geq m_1$.

Note also that a simple function is a linear combination
of characteristic functions of clopen sets, hence there are 
$\xi_1,..., \xi_{k-1}<\xi<\omega_1$ and $m_2$ such that
preimages under $f$ of each of its finite rational values belong to the
subalgebra of ${\mathcal A}_{\xi+1}$  generated by
$V_t$'s for $|t|<m_2$ and $A_{\xi_1,j},..., A_{\xi_{k-1,j}}, A_{\xi,j}$ for $1\leq j\leq N$.
Now let $n\geq m_1, m_2$ be such that for every
$1\leq i<k$ there is $1\leq j\leq N$ such that  $V_{x_{\xi}|m}\subseteq A_{\xi_i,j}$
which can be 
obtained by the property (5) (of Definition \ref{splittingFamily}) of $A_\xi$'s and $\eta=\xi_i$.

It follows that $f$ is constant on $A_{\xi,j}\cap V_{x_{\xi}|m}$ for every $1\leq j\leq N$.
Let $q_1',..., q_{N}' \in \mathbb{Q}$ be the corresponding
values and note that $|q_l'-q_{N}'|\leq 2||f||$ for any $1\leq l\leq N$. So, by conditions (2) and (3) (of Definition \ref{splittingFamily}) of $A_{\xi,j}$'s
we have
$$f=[f|(K\setminus V_{x_{\xi}|m}) 
+ q_{N}'\chi_{V_{x_{\xi}|m}}]+ \sum_{1\leq l<N}(q_l'-q_{N}') \chi_{A_{\xi,l}\cap V_{x_{\xi}|m}}.$$
Note that $f|(K\setminus V_{x_{\xi}|m})$
belongs to $C_\xi$ by Lemma \ref{lemma3}, and so
$$f=h+\sum_{1\leq l<N}q_l\chi_{A_{\xi,l}\cap V_{x_{\xi}|m}},\ 
\max_{1\leq l<N}|q_l||\mu|(V_{x_{\xi}|m}\setminus R_{\xi})\leq{\varepsilon\over{2}}$$
where $q_l=q_l'-q_{N}'$ and $h\in C_\xi$.
Hence the inductive assumption for $\varepsilon/2$ can be used, which completes the proof
of the lemma.
\end{proof} 

\begin{definition}\label{balancedFamily}
We say that an $N$-splitting family $(A_{\xi,i}: \xi<\omega_1, i\in [N])$ 
is balanced if it satisfies the following additional condition:
\begin{enumerate}[$(1)$]\setcounter{enumi}{5}
\item for all distinct $\xi, \eta \in \omega_1$ and all $j \in [2n]$,
$$|\{i \in \{1, 3, \dots, 2n-1\}: (x_\eta, i) \in A_{\xi,j}\}| = |\{i \in \{2, 4, \dots, 2n\}: (x_\eta, i) \in A_{\xi,j}\}|.$$
\end{enumerate}
\end{definition}

\begin{lemma}\label{property6}
Suppose that $n\in\N$ and that $K_{2n}$ is an unordered ${2n}$-split Cantor set,
where the $N$-splitting family $(A_{\xi,i}: \xi<\omega_1, i\in [2n])$ is balanced.
Then we have that:
\begin{enumerate}[(a)]
\item $K_{2n}^{n+1}$ contains an uncountable discrete subspace;
\item there is an uncountable biorthogonal system in $C(K_{2n})$ 
with $2n$-supported functionals. 
\end{enumerate}
\end{lemma}
\begin{proof}
To prove (a), let us show that the subset $\{((x_\xi,1), (x_\xi,2), (x_\xi,4), \dots , (x_\xi, 2n)): \xi < \omega_1\}$  of $K_{2n}^{n+1}$ is relatively discrete.

Let $U_\xi = A_{\xi, 1} \times A_{\xi, 2} \times A_{\xi, 4} \times \dots \times A_{\xi, 2n}$, which clearly is
an open neighbourhood of $((x_\xi,1), (x_\xi,2), (x_\xi,4), \dots , (x_\xi, 2n))$. Now, fix distinct $\xi, \eta < \omega_1$ and let us prove that
$((x_\eta,1), (x_\eta,2), (x_\eta,4), \dots , (x_\eta, 2n)) \notin U_\xi$. 

By contradiction, suppose that $((x_\eta,1), (x_\eta,2), (x_\eta,4), \dots , (x_\eta, 2n)) \in U_\xi$, that
is, $(x_\eta, j) \in A_{\xi,j}$ for each $j =1,2,4, \dots, 2n$. By condition \ref{balancedFamily}.(6), we have that for each $j \in [2n]$, 
$$|\{i \in \{1, 3, \dots, 2n-1\}: (x_\eta, i) \in A_{\xi,j}\}| = |\{i \in \{2, 4, \dots, 2n\}: (x_\eta, i) \in A_{\xi,j}\}|.$$
Hence, each set $A_{\xi, 2}, A_{\xi, 4}, \dots, A_{\xi, 2n}$ must contain at least one of the $(x_\eta, 1), (x_\eta, 3),$ $\dots,(x_\eta, 2n-1)$.
By the disjointness of the $A_{\xi, j}$'s (Property (2) of Definition \ref{splittingFamily}), $(x_\eta, 1)$ has to be in one of the sets
$A_{\xi, 2}, A_{\xi, 4}, \dots, A_{\xi, 2n}$. But by our assumption, $(x_\eta,1) \in A_{\xi, 1}$ and again by the disjointness of the
$A_{\xi,j}$'s, this is a contradiction.

To show (b), for each $\xi<\omega_1$, let $f_\xi = \chi_{A_{\xi,2n}}$ and 
$$\mu_\xi = \sum_{k=1}^n (\delta_{(x_\xi, 2i)} - \delta_{(x_\xi, 2i-1)})$$
and note that $(f_\xi, \mu_\xi)_{\xi <\omega_1} \subseteq C(K_{2n}) \times M(K_{2n})$.
Let us prove that this is a biorthogonal system.

For each $\xi < \omega_1$, since $(x_\xi, i) \in A_{\xi, i}$ and these sets are disjoint (Property (2) of Definition \ref{splittingFamily}), we get that
$$\mu_\xi(f_\xi) = \sum_{k=1}^n (\delta_{(x_\xi, 2k)} - \delta_{(x_\xi, 2k-1)})(\chi_{A_{\xi,2n}})$$ 
$$= \sum_{k=1}^n (\chi_{A_{\xi,2n}}((x_\xi, 2k)) - \chi_{A_{\xi,2n}}((x_\xi, 2k-1)))
= \chi_{A_{\xi,2n}}((x_\xi, 2n)) = 1.$$

On the other hand, for distinct $\xi, \eta < \omega_1$, by Property (6),
we have that for all $j \in [2n]$,
$$|\{i \in \{1, 3, \dots, 2n-1\}: (x_\eta, i) \in A_{\xi,j}\}| = |\{i \in \{2, 4, \dots, 2n\}: (x_\eta, i) \in A_{\xi,j}\}|.$$
Hence,
$$\begin{array}{cl}
\mu_\xi(f_\eta) & = \displaystyle{\sum_{k=1}^n} (\delta_{(x_\xi, 2k)} - \delta_{(x_\xi, 2k-1)})(\chi_{A_{\eta,2n}}) \\
&= \displaystyle{\sum_{k=1}^n} (\chi_{A_{\eta,2n}}((x_\xi, 2k)) - \chi_{A_{\eta,2n}}((x_\xi, 2k-1)))\\
&= \displaystyle{\sum_{k=1}^n} \chi_{A_{\eta,2n}}((x_\xi, 2k)) - \displaystyle{\sum_{k=1}^n} \chi_{A_{\eta,2n}}((x_\xi, 2k-1))\\
& = |\{i \in \{2, 4, \dots, 2n\}: (x_\xi, i) \in  A_{\eta,2n}\}|  - |\{i \in \{1, 3, \dots, 2n-1\}: (x_\xi, i) \in  A_{\eta,2n}\}|\\
& =0,
\end{array}$$
concluding that $(f_\xi, \mu_\xi)_{\xi <\omega_1} \subseteq C(K_{2n}) \times M(K_{2n})$ is a biorthogonal system.
\end{proof}

\section{The generic construction}
This section is devoted to a generic construction of
an unordered $2n$-split Cantor set which exhibits quite
random features. This type of uncountable structures 
was first investigated systematically in \cite{shelah2}.
One can describe this random behavior as: in any
uncountable sequence of finite substructures we have two which are
related as we wish (up to constrains). We fix an
uncountable sequence $(x_\xi:\xi<\omega_1)\subseteq 2^\omega$
consisting of distinct elements.

\begin{definition}
Let $\mathbb{P}$ be the forcing formed by conditions 
$$p=(F_p, n_p, (f^p_\xi: \xi \in F_p)),$$
where:
\begin{enumerate}[1.]
\item $F_p \in [\omega_1]^{<\omega}$;
\item $n_p \in \omega$ is such that for all $\xi\neq \eta$ in $F_p$, $x_\xi|n_p \neq x_\eta|n_p$;
\item for all $\xi \in F_p$, 
$$f^p_\xi: 2^{n_p} \setminus \{x_\xi|n_p\} \rightarrow [2n]^{[2n]} \times [F_p \cap (\xi+1)]$$
is such that 
\begin{enumerate}[a)]
\item if $f^p_\xi(s)=(\varphi, \xi)$, then $\varphi$ is a constant function;
\item if $f^p_\xi(s) = (\varphi, \eta)$ for some $\eta < \xi$, then 
$$\forall j \in [2n] \quad |\varphi^{-1} (j) \cap \{1, 3, 5, \dots, 2n-1\} | = |\varphi^{-1}(j) \cap \{2, 4, \dots, 2n\}|.$$
\end{enumerate}
\end{enumerate}
We put $q \leq p$ if $F_q \supseteq F_p$, $n_q \geq n_p$ and for all $\xi \in F_p$, all $s \in 2^{n_q} \setminus \{x_\xi|n_q\}$ and all $t \in 2^{n_p} \setminus \{x_\xi|n_p\}$,
$$t \subseteq s \Rightarrow f^p_\xi(t) = f^q_\xi(s).$$
\end{definition}

Intuitively, we are,
of course, trying to build a $2n$-split
Cantor set which is determined by the
choice of the balanced $2n$-splitting family formed by $A_{\xi,i}$'s. Thus the
coordinate $f^p_\xi(s)$ describes the behavior of $A_{\xi, i}$'s on $V_s$.
The formal description is the subject of Definition \ref{axis}.
The value $f^p_\xi(s)=(\varphi, \xi)$, where $\varphi$ has to be a constant function,  say equal to $i$, means that the entire $V_s$ is included in
$A_{\xi, i}$. The value $f^p_\xi=(\varphi, \eta)$ for some $\eta<\xi$ means
that $A_{\xi, i}$'s divide $V_s$ as coded by $\varphi$ i.e.,
$A_{\eta,j}\cap V_s\subseteq A_{\xi,\varphi(j)}$ for each $j\in [N]$.
Note that a condition $p\in\mathbb{P}$ does not carry any information
about the behavior of $A_{\xi,i}$'s on $V_{x_\xi|n_p}$, other than
$(x_\xi, i)\in A_{\xi, i}$. This is the 
degree of freedom we have and which can be controlled by passing to
an appropriate extension $q\leq p$. Condition (b) is to guarantee
that the family of $A_{\xi,i}$'s is balanced, that is, that is satisfies
property (6) of Definition \ref{balancedFamily}.

\begin{lemma}\label{lemaDensidade}
The following subsets of $\mathbb{P}$ are dense in $\mathbb{P}$:
\begin{enumerate}[(i)]
\item $\{p \in \mathbb{P}: n_p \geq k\}$, for some fixed $k \in \mathbb{N}$;
\item $\{p \in \mathbb{P}: \xi \in F_p \}$, for some fixed $\xi <\omega_1$.
\end{enumerate}
\end{lemma}
\begin{proof}
For $(i)$, fix $k \in \mathbb{N}$ and let $p =(F_p, n_p, (f^p_\xi: \xi \in F_p)) \in \mathbb{P}$. If $n_p<k$, define $q = (F_q, n_q, (f^q_\xi: \xi \in F_q))$
by putting $F_q = F_p$, $n_q = k$ and for each $\xi \in F_q = F_p$, $f^q_\xi$ is any function satisfying condition 3 of the definition of the forcing such that 
$f^q_\xi(t)= f^p_\xi(t|n_p)$, if $t|n_p \in 2^{n_p} \setminus \{x_\xi|n_p\}$; for example, let
$$f^q_\xi(t) = \left\{\begin{array}{cl}
f^p_\xi(t|n_p)& \text{if }t|n_p \in 2^{n_p} \setminus \{x_\xi|n_p\},\\
(\varphi, \xi) & \text{otherwise,}
\end{array}\right.$$
where $\varphi$ is the constant function equal to $1$. It is easy to see that $q \in \mathbb{P}$ and $q \leq p$.

For $(ii)$, fix $\xi<\omega_1$ and let $p =(F_p, n_p, (f^p_\xi: \xi \in F_p)) \in \mathbb{P}$. By $(i)$, we may assume that $n_p$ is such that
$x_\eta|n_p \neq x_\xi|n_p$ for all $\eta \in F_p$. Define $q = (F_q, n_q, (f^q_\xi: \xi \in F_q))$
by putting $F_q = F_p\cup \{\xi\}$, $n_q = n_p$, $f^q_\eta = f^p_\eta$ for each $\eta \in F_p$, and
$f^q_\xi$ is any function satisfying condition 3 of the definition of the forcing;  for example, let
$f^q_\xi(t) = (\varphi, \xi)$, where $\varphi$ is the constant function equal to $1$. It is easy to see that $q \in \mathbb{P}$ and $q \leq p$.
\end{proof}

\begin{definition}\label{axis}
Given a $\mathbb{P}$-generic filter $G$ over a model $V$, we define the family $\{A_{\xi, j}: \xi \in \omega_1, j \in [2n]\}$ as follows:
for each $\xi \in \omega_1$ and each $j \in [2n]$, let
$$\begin{array}{ll}
A_{\xi, j} = & \{(x_\xi, j)\} \cup \bigcup \{V_s: \exists p \in G, f^p_\xi(s) = (\varphi, \xi) \text{ and } \varphi \text{ is the constant function equal to }j\}\\
& \cup \bigcup \{V_s \cap A_{\eta, i}: \exists p \in G, f^p_\xi(s) = (\varphi, \eta), \text{ for some $\eta \neq \xi$ and } \varphi(i)=j\}.
\end{array}$$
\end{definition}

The following lemma follows directly from the above definition.

\begin{lemma}\label{PropBas}
Given $p \in G$, $\xi \in F_p$ and $s \in 2^{n_p} \setminus \{x_\xi|n_p\}$, we have that:
\begin{enumerate}[(a)]
\item if $f^p_\xi(s) = (\varphi, \xi)$, then $V_s \subseteq A_{\xi, j}$ for $j = \varphi(1)$;
\item if $f^p_\xi(s) = (\varphi, \eta)$ for some $\eta < \xi$, then $\forall i \in [2n]$, $V_s \cap A_{\eta, i} \subseteq A_{\xi, \varphi(i)}$. $\square$
\end{enumerate}
\end{lemma}

Notice that in case $f^p_\xi(s) = (\varphi, \xi)$,  $\varphi$ is the constant function equal to $j$, 
so that we could have taken $j = \varphi(i)$ for any $i \in [2n]$.

Let us now check that the family $\{A_{\xi, j}: \xi \in \omega_1, j \in [2n]\}$ has the desired properties.

\begin{theorem}\label{listaPropriedades}
The family $\{A_{\xi, j}: \xi \in \omega_1, j \in [2n]\}$ is a balanced $2n$-splitting family.
\end{theorem}
\begin{proof}
Let us prove that the family satisfies conditions \ref{splittingFamily}.(1) - (5) and \ref{balancedFamily}.(6). 

\noindent \textbf{(1).} It follows directly from the definition of $A_{\xi, j}$.

\noindent \textbf{(2).} We prove it by induction on $\xi$. First notice that by the definition of the forcing $\mathbb{P}$, 
$$\forall p \in \mathbb{P} \ \forall \xi \in F_p \ \forall s \in dom f^p_\xi \quad R_\xi\cap V_s=\emptyset,$$
since $x_\xi|n_p \notin dom f^p_\xi$. Then, $(x_\xi, j_1) \in A_{\xi, j_2}$ iff $j_1=j_2$. 

Now, fix $\xi < \omega_1$ and suppose we have that $A_{\eta, i}$ are pairwise disjoint for each fixed $\eta<\xi$.
Suppose there is $x \in A_{\xi, j_1} \cap A_{\xi, j_2}$ for some distinct $j_1, j_2 \in [2n]$. By the above observation,
$x \neq (x_\xi, j)$ for any $j \in [2n]$.

By the definition of $A_{\xi, j_k}$, for each $k \in \{1,2\}$ 
there are $p_k \in G$ 
and 
$s_k \in dom f^{p_k}_\xi$ such that $x \in V_{s_k}$ and 
$$\text{either }  f^{p_k}_\xi(s_k) = (\varphi_k, \xi) \text{ and } \varphi_k \text{ is the constant function equal to } j_k$$
$$\text{or } f^{p_k}_\xi(s_k) = (\varphi_k, \eta_k) \text{ for some $\eta_k < \xi$ and } x \in A_{\eta_k, i} \text{ for some } i \in \varphi_k^{-1}(j_k).$$

Let $p \in G$ be such that $p\leq p_1,p_2$ and let $t \in 2^{n_p} \setminus \{x_\xi|n_p\}$ be such that $x \in V_t$. Then, $t \supseteq s_k$ since $x \in V_{s_k}$ and hence,
by the definition of extension in $\mathbb{P}$, $f^{p_1}_\xi(s_1) = f^{p}_\xi(t) = f^{p_2}_\xi(s_2)$, so that $\varphi_1 = \varphi_2$. 

Now, if $f^{p}_\xi(t) = (\varphi, \xi)$, this would mean that $\varphi_1$ and $\varphi_2$ are both constant equal to $j_1$ and $j_2$, contradicting the hypothesis
that $j_1 \neq j_2$. Otherwise, if $f^{p}_\xi(t) = (\varphi, \eta)$, for some $\eta < \xi$, 
we would get that $x \in A_{\eta, i_k}$ for some $i_k \in \varphi^{-1}(j_k)$. By the inductive hypothesis we get that
$i_1 = i_2 \in \varphi^{-1}(j_1) \cap \varphi^{-1}(j_2)$, which implies that $j_1 = j_2$, contradicting again the hypothesis that those are 
distinct. 

This concludes the proof that the family satisfies condition (2) of Definition \ref{splittingFamily}.

\noindent \textbf{(3).} Again we prove it by induction on $\xi$. So, let $\xi<\omega_1$, suppose $K = A_{\eta,1} \cup \dots \cup A_{\eta, 2n}$ for any
$\eta<\xi$ and let $x \in K$.

If $x = (x_\xi, i)$ for some $i \in [2n]$, then $x \in A_{\xi, i}$ by definition.

By Lemma \ref{lemaDensidade}, let $p \in G$ be such that
 $x \in V_s$ for some $s \in  2^{n_p} \setminus \{x_\xi|n_p\}$. 

If $f^p_\xi(s)  = (\varphi, \xi)$, by Lemma \ref{PropBas}.(a) we get that
$V_s \subseteq A_{\xi, \varphi(1)}$, which guarantees that $x \in A_{\xi, \varphi(1)}$. 

Otherwise, if $f^p_\xi(s)  = (\varphi, \eta)$, 
for some $\eta < \xi$, by the inductive hypothesis, let $i \in [2n]$ be such that $x \in A_{\eta, i}$. Then, by Lemma \ref{PropBas}.(b),
$V_s \cap A_{\eta, i} \subseteq A_{\xi, \varphi(i)}$, which implies that $x \in A_{\xi, \varphi(i)}$ and
concludes the proof of condition (3) of Definition \ref{splittingFamily}.

\noindent \textbf{(4).} Fix $\eta<\xi<\omega_1$ and $i \in [2n]$. By Lemma \ref{lemaDensidade}, let $p \in G$ be such that
$\xi, \eta \in F_p$ and $x_\eta|n_p \neq x_\xi|n_p$. 

If $f^p_\xi(x_\eta|n_p) = (\varphi, \xi)$, by Lemma \ref{PropBas}.(a) we get that
$V_{x_\eta|n_p} \subseteq A_{\xi, \varphi(1)}$ (and in particular $V_{x_\eta|n_p} \cap A_{\eta, i} \subseteq V_{x_\eta|n_p} \cap A_{\xi, \varphi(1)}$).

If $f^p_\xi(x_\eta|n_p) = (\varphi, \eta)$, for some $\eta < \xi$, then, by Lemma \ref{PropBas}.(b),
$V_{x_\eta|n_p} \cap A_{\eta, i} \subseteq A_{\xi, \varphi(i)}$ (and in particular $V_{x_\eta|n_p} \cap A_{\eta, i} \subseteq V_{x_\eta|n_p} \cap A_{\xi, \varphi(i)}$) 
and we're done with condition (4) of Definition \ref{splittingFamily}.

\noindent \textbf{(5).} Let us prove this by induction on $\xi<\omega_1$. Let $\xi < \omega_1$ and 
$x \in 2^\omega \setminus \{x_\eta: \eta\leq \xi\}$.

If $x = x_\eta$ for some $\eta > \xi$, by Lemma \ref{lemaDensidade} there
is t $p \in G$ be such that $\xi, \eta \in F_p$. Otherwise, 
if $x \in 2^\omega \setminus \{x_\eta: \eta<\omega_1\}$, by Lemma \ref{lemaDensidade} there is $p \in G$ be such that $\xi \in F_p$ 
and $x|n_p \neq x_\xi|n_p$. In both cases, put $s = x|n_p$.

If $f^p_\xi(s) = (\varphi, \xi)$, then, by Lemma \ref{PropBas}.(a), $V_{s} \subseteq A_{\xi, \varphi(1)}$. 

If $f^p_\xi(s) = (\varphi, \eta')$
for some $\eta' \in F_p \cap \xi$, by the inductive hypothesis, there is $k \in \mathbb{N}$ and $i \in [2n]$ such that $V_{x|k} \subseteq A_{\eta',i}$.
By Lemma \ref{lemaDensidade}, let $q \in G$ be such that $q \leq p$ and $n_q \geq k$. Putting $t=x|n_q$, we get that
$V_t \subseteq V_{x|k} \subseteq A_{\eta',i}$ and
$f^q_\xi(t) = f^p_\xi(s) = (\varphi, \eta')$, since $t \supseteq s$. This implies by Lemma \ref{PropBas}.(b) that $V_t= 
V_t \cap A_{\eta', i} \subseteq A_{\xi, \varphi(i)}$, which concludes the proof of condition (5) of Definition \ref{splittingFamily}.  

Hence, the family formed by $A_{\xi, i}$'s is a $2n$-splitting family.

\noindent \textbf{(6).} Let us prove this by induction on $\xi<\omega_1$. 
So, fix $\xi<\omega_1$ and suppose we know that for all $\zeta <\xi$, all $\eta \neq \zeta$ and all $j \in [2n]$,
$$|\{i \in \{1, 3, \dots, 2n-1\}: (x_\eta,i) \in A_{\zeta,j}\}| = |\{i \in \{2, 4, \dots, 2n\}: (x_\eta,i) \in A_{\zeta,j}\}|.$$

Now, fix $\eta \neq \xi$. Let $p \in G$ be such that $\xi, \eta \in F_p$, so that $x_\eta|n_p \in dom f^p_\xi$. 

If $f^p_\xi(x_\eta|n_p) = (\varphi, \xi)$, then, by Lemma \ref{PropBas}.(a), $V_{x_\eta|n_p} \subseteq A_{\xi, \varphi(1)}$, which implies that
$(x_\eta, i) \in A_{\xi, \varphi(1)}$ for all $i \in [2n]$. By the disjointness of the $A_{\xi,i}$'s (3)
and condition (6) 
of Definition \ref{balancedFamily}
holds both for $A_{\xi, \varphi(1)}$ 
(which contains all $(x_\eta,i)$) and for $A_{\xi, j}$, $j \neq \varphi(1)$ (which contain no $(x_\eta,i)$).

If $f^p_\xi(x_\eta|n_p) = (\varphi, \zeta)$ for some $\zeta < \xi$ in $F_p$, then for all $i \in [2n]$, $V_{x_\eta|n_p} \cap A_{\zeta, i} \subseteq A_{\xi, \varphi(i)}$.
This means that each $A_{\xi, j}$ contains exactly those $(x_\eta,k)$ which are in $A_{\zeta,i}$ for some $i \in \varphi^{-1}(j)$. In particular, we have that
$$\{k \in \{1, 3, \dots, 2n-1\}: (x_\eta, k) \in A_{\xi,j}\}$$ 
$$= \{k \in \{1, 3, \dots, 2n-1\}: (x_\eta,k) \in A_{\zeta,i} \text{ for some }i \in \varphi^{-1}(j)\}$$
$$= \bigcup_{i \in \varphi^{-1}(j)}  \{k \in \{1, 3, \dots, 2n-1\}: (x_\eta,k) \in A_{\zeta,i}\}$$
and
$$\{k \in \{2, 4, \dots, 2n\}: (x_\eta,k) \in A_{\xi,j}\}$$
$$ = \{k \in \{2, 4, \dots, 2n\}: (x_\eta,k) \in A_{\zeta,i} \text{ for some }i \in \varphi^{-1}(j)\}$$
$$= \bigcup_{i \in \varphi^{-1}(j)}  \{k \in \{2, 4, \dots, 2n\}: (x_\eta,k) \in A_{\zeta,i}\}.$$

Let us now consider two cases:

If $\eta = \zeta$, since $(x_\eta,k) \in A_{\eta,k}$, we get that 
$$\{k \in \{1, 3, \dots, 2n-1\}: (x_\eta,k) \in A_{\xi,j}\} = \{k \in \{1, 3, \dots, 2n-1\}: k \in \varphi^{-1}(j)\}$$
and
$$\{k \in \{2, 4, \dots, 2n\}: (x_\eta,k) \in A_{\xi,j}\} = \{k \in \{2, 4, \dots, 2n\}: k \in \varphi^{-1}(j)\}.$$
By property 3.b) of the definition of the partial ordering, the sets on the right-hand side of these two equalities have 
same size, which guarantees that
$$|\{k \in \{1, 3, \dots, 2n-1\}: (x_\eta,k) \in A_{\xi,j}\}| = |\{k \in \{2, 4, \dots, 2n\}: (x_\eta,k) \in A_{\xi,j}\}|,$$
concluding the proof in this case.

If $\eta \neq \zeta$, by the inductive hypothesis we know that for all $i \in [2n]$,
$$|\{k \in \{1, 3, \dots, 2n-1\}: (x_\eta,k) \in A_{\zeta,i}\}| = |\{k \in \{2, 4, \dots, 2n\}: (x_\eta,k) \in A_{\zeta,i}\}|.$$
Hence,
$$|\{k \in \{1, 3, \dots, 2n-1\}: (x_\eta, k) \in A_{\xi,j}\}|$$
$$ = |\bigcup_{i \in \varphi^{-1}(j)}  \{k \in \{1, 3, \dots, 2n-1\}: (x_\eta,k) \in A_{\zeta,i}\}|$$
$$ = |\bigcup_{i \in \varphi^{-1}(j)}  \{k \in \{2, 4, \dots, 2n\}: (x_\eta,k) \in A_{\zeta,i}\}|$$
$$= |\{k \in \{2, 4, \dots, 2n\}: (x_\eta,k) \in A_{\xi,j}\}|,$$
which concludes the proof of condition (6) of Definition \ref{balancedFamily}, that is,
the family of $A_{\xi,i}$'s is a balanced $2n$-splitting family.
\end{proof}

\begin{proposition}\label{propAmalgamacao}
Let $p_1 = (F_1, n_1, (f^1_\xi: \xi \in F_1))$ and $p_2= (F_2, n_2, (f^2_\xi: \xi \in F_2))$ be conditions of $\mathbb{P}$ such that:
\begin{itemize}
\item $F_1 \cap F_2 < F_1 \setminus F_2<F_2 \setminus F_1$;
\item $n_1=n_2 = n$;
\item there is an order-preserving bijection $e: F_1 \rightarrow F_2$ such that 
\begin{itemize}
\item for all $\xi \in F_1$, $x_\xi|n = x_{e(\xi)}|n$;
\item for all $\xi \in F_1$ and all $s \in 2^{n_1} \setminus \{x_\xi|n_1\} (= 2^{n_2} \setminus \{x_{e(\xi)}|n_2\})$,
$$f^2_{e(\xi)}(s) = (\varphi, e(\eta)) \quad  \text{where } f^1_\xi(s) = (\varphi, \eta).$$
\end{itemize}
\end{itemize}
Then, given $(\epsilon_\xi: \xi \in F_1 \setminus F_2) \subseteq [2n]^{[2n]}$ such that for all $\xi \in F_1 \setminus F_2$
$$\forall j \in [2n] \quad |\epsilon_\xi^{-1} (j) \cap \{1, 3, 5, \dots, 2n-1\}| = |\epsilon_\xi^{-1}(j) \cap \{2, 4, 6, \dots, 2n\}|$$
and given constant functions $(\delta_\xi: \xi \in F_1 \setminus F_2) \subseteq [2n]^{[2n]}$, there is $q \leq p_1, p_2$, $q \in \mathbb{P}$ such that
\begin{equation}\label{eq1}
\forall \xi \in F_1 \setminus F_2 \qquad f^q_\xi(x_{e(\xi)}|n_q) = (\delta_\xi, \xi)
\quad \text{and} \quad 
f^q_{e(\xi)}(x_{\xi}|n_q) = (\epsilon_\xi, \xi).
\end{equation}
\end{proposition}

\begin{proof}
Let $q = (F_q, n_q, (f^q_\xi: \xi \in F_q))$ be defined
as follows: let $F_q = F_1 \cup F_2$; let $n_q \in \mathbb{N}$ be such $n_p \leq n_q$ and for all $\xi <\eta \in F_q$, $x_\xi|n_q \neq x_\eta|n_q$; 
for each $\xi \in F_q$ and $t \in 2^{n_q} \setminus \{x_\xi|n_q\}$, let 
$$f^q_\xi(t) =  \left\{
              \begin{array}{llr}
					f^1_\xi(t|n) & \text{if } \xi \in F_1 \text{ and } t|n \neq x_\xi|n & (\text{Case 1.})\\
					f^2_\xi(t|n) & \text{if } \xi \in F_2 \text{ and } t|n \neq x_\xi|n & (\text{Case 2.})\\
					(\delta_\xi, \xi) & \text{if } \xi \in F_1 \text{ and } t|n = x_\xi|n & (\text{Case 3.})\\
					(\epsilon_{e^{-1}(\xi)}, e^{-1}(\xi)) & \text{if } \xi \in F_2 \setminus F_1 \text{ and } t|n = x_\xi|n & (\text{Case 4.})
              \end{array} \right.$$

$f^q_\xi$ is well-defined since $e(\xi)=\xi$ whenever $\xi \in F_1 \cap F_2$, so that $f^1_\xi(s) = f^2_{e(\xi)}(s) =f^2_\xi(s)$ for $s \in 2^{n} \setminus \{x_\xi|n\}$.

Let us now prove that $q \in \mathbb{P}$. Conditions 1 and 2 follow directly from the definition of $F_q$ and $n_q$. 

To prove that $q$ satisfies condition 3, fix $\xi \in F_q$
and $t \in 2^{n_q} \setminus \{x_\xi|n_q\}$. In Case 1 (resp. Case 2), both conditions 3.a) and 3.b) follow from the fact that $p_1$ (resp. $p_2$) is in $\mathbb{P}$. 

In Case 3, we only have to check condition 3.a), which is guaranteed by the fact that $(\delta_\xi: \xi \in F_1\setminus F_2) \subseteq [2n]^{[2n]}$ are assumed to be constant. 

Similarly, in Case 4, we only have to check condition 3.b), which is guaranteed by the fact that $(\epsilon_\xi: \xi \in F_1 \setminus F_2) \subseteq [2n]^{[2n]}$ 
are assumed to be as needed. 

Let us now prove that $q \leq p_1 , p_2$. Trivially, $F_1, F_2 \subseteq F_q$ and $n_1, n_2 \leq n_q$. 

Given $\xi \in F_q$, $s \in 2^n \setminus \{x_\xi|n\}$ and $t \in 2^{n_q} \setminus \{x_\xi|n_q\}$ such that $s \subseteq t$, let $k \in \{1,2\}$
be such that $\xi \in F_k$ and notice that we are in cases 1 or 2, since $t|n=s$. Therefore, $f^q_\xi(t) = f^k_\xi(t|n) = f^k_\xi(s)$, which
concludes that $q \leq p_1, p_2$.

Finally, notice that the definition of $f^q_\xi(t)$ in cases 1 or 2 imply (\ref{eq1}).
\end{proof}

\begin{theorem}
$\mathbb{P}$ is c.c.c.
\end{theorem}
\begin{proof}
For each $\alpha<\omega_1$, let $p_\alpha=(F_\alpha, n_\alpha, (f^\alpha_\eta)_{\eta\in F_\alpha}) \in
\mathbb{P}$. 

By the $\Delta$-system Lemma, we can assume that
$(F_\alpha)_{\alpha<\omega_1}$ forms a $\Delta$-system with root $\Delta$ such that for every $\alpha < \beta < \omega_1$,
\begin{itemize}
\item $\Delta < F_\alpha \setminus \Delta < F_\beta \setminus \Delta$ and $|F_\alpha|=|F_\beta|$.
\end{itemize}
Since each $n_\alpha \in \mathbb{N}$, we can suppose that for every $\alpha < \beta < \omega_1$, 
\begin{itemize}
\item $n_\alpha = n_\beta = n$.
\end{itemize}
Also, we may assume that if $e_{\alpha\beta}:F_\alpha \rightarrow F_\beta$ is the order-preserving bijective function, then 
\begin{itemize}
\item for all $\xi \in F_\alpha$, $x_\xi|n = x_{e_{\alpha \beta}(\xi)}|n$ (since both belong to $2^n$);
\item for all $\xi \in F_\alpha$ and all $s \in 2^{n} \setminus \{x_\xi|n\}$, 
$$f^\beta_{e_{\alpha \beta}(\xi)}(s)= (\varphi, e_{\alpha \beta}(\eta)),  \text{ where } f^\alpha_\xi(s) = (\varphi, \eta).$$
\end{itemize}

Now, fix $\alpha<\beta<\omega_1$. Note that $p_\alpha$ and $p_\beta$ satisfy the hypothesis of Proposition
\ref{propAmalgamacao}. Let, for $\xi \in F_\beta \setminus \Delta$, $\epsilon_\xi$ be any function satisfying the 
condition 3 of the definition of the forcing (for example, $\epsilon_\xi$ constant equal to $1$); and for $\xi \in F_\alpha \setminus \Delta$, $\delta_\xi \in
[2n]^{[2n]}$ be any constant function. Then, by Proposition \ref{propAmalgamacao}, there is $q\leq p_\alpha, p_\beta$ in
$\mathbb{P}$, which concludes the proof.
\end{proof}

\begin{theorem}\label{generictheorem}
Let $n\geq 1$ be a natural number.
It is consistent that there is a compact Hausdorff totally disconnected space $K$ which is an unordered $2n$-split Cantor set
corresponding to a balanced $2n$-splitting family $(A_{\xi,i}: \xi<\omega_1, i \in [2n])$ 
such that given any collection of pairwise disjoint sets $E_\alpha=\{\xi_\alpha^1,..., \xi_\alpha^k\}\subseteq\omega_1$ 
for $\alpha<\omega_1$,  given $\epsilon:[k]\times [2n] \rightarrow[2n]$ such that 
$$|\{l\in\{1,3,5,...,2n-1\}: \epsilon(i,l)=j\}|= |\{l\in\{2,4,6,...,2n\}: \epsilon(i,l)=j\}|$$
and given $\delta:[k]\rightarrow [n]$, there are $\alpha<\beta$ such that for all $1\leq i\leq k$, 
$$R_{\xi_\beta^i}\subseteq A_{\xi_\alpha^i,\delta(i)}$$
and
$$(x_{\xi^i_\alpha}, l)\in  A_{\xi_\beta^i, \epsilon(i,l)}.$$
\end{theorem}
\begin{proof}
By Theorem \ref{listaPropriedades}, $\mathbb{P}$ forces that $(A_{\xi,i}: \xi<\omega_1, i \in [2n])$ as in
Definition \ref{axis} is a balanced $2n$-splitting family. By Proposition \ref{splitCantor}, we get that the corresponding
unordered $2n$-split Cantor set is a compact, Hausdorff, totally disconnected space. Let us now prove the remaining
desired property. 

In $V$, suppose $(\dot{E}_{\alpha})_{\alpha<\omega_1}$ and $(\dot{\xi}^i_\alpha)_{\alpha<\omega_1, 1 \leq i \leq k}$ are sequences of names
such that $\mathbb{P}$ forces that $\dot{E}_\alpha = \{\dot{\xi}^1_\alpha < \dots < \dot{\xi}^k_\alpha\}$ and $(\dot{E}_\alpha)_{\alpha<\omega_1}$ is pairwise disjoint.

For each $\alpha<\omega_1$, let $p_\alpha=(F_\alpha, n_\alpha, (f^\alpha_\eta)_{\eta\in F_\alpha}) \in
\mathbb{P}$, $\xi^1_\alpha, \dots, \xi^k_\alpha \in \omega_1$ and $E_\alpha, \dots,
E_\alpha \subseteq \omega_1$ be finite such that
       $$p_\alpha \Vdash \forall 1\leq i \leq k \quad \dot{\xi}^i_\alpha =
       \check{\xi}^i_\alpha    \text{ and } \dot{E}_\alpha = \check{E}_\alpha.$$

By Lemma \ref{lemaDensidade}, we can assume without loss of generality that
for all $\alpha<\omega_1$, $E^\alpha_i \subseteq F_\alpha$.

By the $\Delta$-system Lemma, we can assume as well that
$(F_\alpha)_{\alpha<\omega_1}$ forms a $\Delta$-system with root $\Delta$ such that for every $\alpha < \beta < \omega_1$,
\begin{itemize}
\item $\Delta < F_\alpha \setminus \Delta < F_\beta \setminus \Delta$ and $|F_\alpha|=|F_\beta|$.
\end{itemize}
Since each $n_\alpha \in \mathbb{N}$, we can suppose that for every $\alpha < \beta < \omega_1$, 
\begin{itemize}
\item $n_\alpha = n_\beta = n$.
\end{itemize}
Also, we may assume that if $e_{\alpha\beta}:F_\alpha \rightarrow F_\beta$ is the order-preserving bijective function, then 
\begin{itemize}
\item for all $\xi \in F_\alpha$, $x_\xi|n = x_{e_{\alpha \beta}(\xi)}|n$ (since both belong to $2^n$);
\item for all $\xi \in F_\alpha$ and all $s \in 2^{n} \setminus \{x_\xi|n\}$, 
$$f^\beta_{e_{\alpha \beta}(\xi)}(s)= (\varphi, e_{\alpha \beta}(\eta)),  \text{ where } f^\alpha_\xi(s) = (\varphi, \eta).$$
\item for all $1 \leq i \leq k$, $e_{\alpha \beta}(\xi^i_\alpha) = \xi^i_\beta$.
\end{itemize}

Finally, we may assume that for all $1 \leq i \leq k$ we have: either
$\xi^i_\alpha = \xi^i_\beta$ for all $\alpha<\beta<\omega_1$; or $\xi^i_\alpha
\notin \Delta$ for all $\alpha<\omega_1$ and actually the second case holds by 
the assumption that $E_\alpha$'s are pairwise disjoint.

Now, fix $\alpha<\beta<\omega_1$. Note that $p_\alpha$ and $p_\beta$ satisfy the hypothesis of Proposition
\ref{propAmalgamacao}. Taking $\epsilon_{\xi^i_\beta} = \epsilon(i, \cdot)$ and $\delta_{\xi^i_\alpha} = \delta(i)$
(and for $\xi \in F_\beta \setminus (\Delta \cup E_\beta)$, any function $\epsilon_\xi$ satisfying the 
condition 3 of the definition of the forcing; and for $\xi \in F_\alpha \setminus (\Delta \cup E_\alpha)$, any constant function $\delta_\xi \in
[2n]^{[2n]}$), by the Proposition \ref{propAmalgamacao}, there is $q\leq p_\alpha, p_\beta$ in
$\mathbb{P}$ such that 
$$\forall \xi \in F_\alpha \setminus \Delta \qquad f^q_\xi(x_{e_{\alpha\beta}(\xi)}|n_q) = (\delta_\xi, \xi) 
\quad \text{and} \quad
f^q_{e_{\alpha\beta}(\xi)}(x_{\xi}|n_q) = (\epsilon_{e_{\alpha\beta}(\xi)}, \xi).$$
In particular, for all $1 \leq i \leq k$, 
$$f^q_{\xi^i_\alpha}(x_{\xi^i_\beta}|n_q) = (\delta(i), \xi^i_\alpha) \quad \text{and}  \quad
f^q_{\xi^i_\beta}(x_{\xi^i_\alpha}|n_q) = (\epsilon(i, \cdot), \xi^i_\alpha).$$

By the definition of $A_{\xi, j}$, we get that for all $1 \leq i \leq k$,
$$R_{\xi^i_\beta} \subseteq A_{\xi^i_\alpha, \delta(i)} \quad \text{and} \quad (x_{\xi^i_\alpha}, l) \in A_{\xi^i_\beta, \epsilon(i, l)},$$
which concludes the proof.
\end{proof}

The fact that $2n$ is even is exploited in the above proof.
It turns out that there cannot be an analogue 
of an unordered $N$-split Cantor set for $N=3$ which
behaves as in Theorem \ref{generictheorem}, as we have the following:

\begin{lemma}\label{lemmathree} 
Let $N\geq 3$ be a natural number.
 Suppose that $K$ is
an unordered $N$-split Cantor set  
corresponding to an $N$-splitting family $(A_{\xi,i}: \xi<\omega_1, i \in [N])$  such that 
given any sequence of distinct ordinals $(\xi_\alpha:\alpha<\omega_1)$
and  $j\in  [N]$,
there are $\alpha<\beta$ such that 
$$R_{\xi_\beta}\subseteq A_{\xi_\alpha,j}.$$
Suppose that $(f_\alpha,\mu_\alpha)_{\alpha<\omega_1}$
is a biorthogonal system such that $f_\alpha=\chi_{A_\alpha}$  for some clopen subset $A_\alpha\subseteq K$ and
$\mu_\alpha=r_\alpha\delta_{(x_{\eta_\alpha},1)}+ s_\alpha\delta_{(x_{\eta_\alpha},2)} +t_\alpha\delta_{(x_{\eta_\alpha},3)}$
for all $\alpha<\omega_1$, for
some reals $r_\alpha, s_\alpha, t_\alpha$  and some sequence
$(\eta_\alpha:\alpha<\omega_1)$. Then there is an uncountable nice biorthogonal system in $C(K)$.
\end{lemma}
\begin{proof}

If there is a biorthogonal system of the form
$(\chi_{A_\alpha},r_\alpha\delta_{y_\alpha})$ for $\alpha<\omega_1$
and $y_\alpha\in K$,  we have that
 $r_\alpha=1$ for all $\alpha<\omega_1$ and that $y_\alpha\not\in A_\beta$ for any
$\beta\not=\alpha$ and $y_\alpha\in A_\alpha$. So $(\chi_{A_{\alpha+1}}, 
\delta_{y_{\alpha+1}}-\delta_{y_\alpha}$), say, for
all  limit ordinals $\alpha$ is a nice biorthogonal system.

If there is a biorthogonal system of the form
$(\chi_{A_\alpha},r_\alpha\delta_{y_\alpha}+s_\alpha\delta_{z_\alpha})$ for
 $\alpha<\omega_1$  and $y_\alpha, z_\alpha\in K$, and $r_\alpha, s_\alpha, r_\alpha+s_\alpha\not=0$,
then $r_\alpha, s_\alpha\not\in A_\beta$ for any $\alpha\not=\beta$ and a similar
argument as above gives a nice biorthogonal system. If $r_\alpha+s_\alpha=0$ and
$r_\alpha, s_\alpha\not=0$,
we may assume that $r_\alpha>0$ and so $s_\alpha=-r_\alpha$. 
It follows from the fact that $( r_\alpha\delta_{y_\alpha}+s_\alpha\delta_{z_\alpha})(\chi_{A_\alpha})=1$ that $r_\alpha=1$ and $s_\alpha=-1$, and so we have 
a nice biorthogonal system.

Hence, without loss of generality, we may assume that $r_\alpha, s_\alpha, t_\alpha \neq 0$ for all $\alpha<\omega_1$.
First let us see that there is an uncountable
$X\subseteq \omega_1$ such that $r_\alpha+s_\alpha+t_\alpha=0$ for
all $\alpha\in X$. If not, then there is an uncountable $X\subseteq \omega_1$
and an $\varepsilon>0$ such that $|r_\alpha+s_\alpha+t_\alpha|>\varepsilon$
for each $\alpha\in X$. 

Now  note that as $\mu_\alpha(\chi_{A_\alpha})=1\not=0$, we have $j\in \{1,2,3\}$ such that
$(x_{\eta_\alpha},j)\in A_\alpha$. We may assume that it is the same $j$ for all $\alpha\in X$.
By the form of the basic neighbourhoods of points $(x_{\eta_\alpha},j)$
we have $s\in 2^m$ for some $m\in \N$  such that $(x_{\eta_\alpha},j)\in V_s\cap A_{\eta_\alpha, j}\subseteq A_\alpha$. We may assume that it is the same $s$ for all $\alpha\in X$.
It follows that  for some $n\in \N$ we have
$s=x_{\eta_\alpha}|n$ for
all $\alpha\in X$ an so that $R_{\eta_\alpha}\subseteq V_s$ for all $\alpha\in X$.
Apply the hypothesis of the lemma and obtain $\alpha<\beta$ both in $X$  such that $R_{\eta_\beta}\subseteq A_{\eta_\alpha,j}$
and we get that $R_{\eta_\beta}\subseteq V_s \cap A_{\eta_\alpha,j}\subseteq A_\alpha$.
This means that $0=\mu_\beta(\chi_{A_\alpha})=r_\beta+s_\beta+t_\beta$ contradicting the choice of $\beta\in X$.
So we may assume that $r_\alpha+s_\alpha+t_\alpha=0$ for all $\alpha<\omega_1$.

For three non-zero numbers whose sum is zero, there cannot be any
subsum which is zero, this means, that for $\alpha\not=\beta$, as 
$\mu_\alpha(A_\beta)=0$, we have that either $\{x_\alpha,y_\alpha, z_\alpha\}\cap A_\beta=\emptyset$
or $\{x_\alpha,y_\alpha, z_\alpha\}\subseteq A_\beta$.
So, to make an uncountable nice biorthogonal system out of points $\{x_\alpha,y_\alpha, z_\alpha\}$
and functions $\chi_{A_{\alpha}}$, we need to find any fixed pair of them which
is separated by ${A_{\alpha}}$ for uncountably many $\alpha$'s.

But $A_\alpha$ must separate some pair as $\mu_\alpha(A_\alpha)=1$,
so choose an uncountable subset $Y$ of $\omega_1$ on which the same pair is separated,
say $x_\alpha \in A_\alpha$ and $z_\alpha \notin A_\alpha$.

Define $\nu_\alpha=\delta_{x_\alpha}-\delta_{z_\alpha}$ and
note that $(\chi_{A_\alpha}, \nu_\alpha)_{\alpha\in Y}$
is an uncountable nice biorthogonal system.
\end{proof}

\section{Biorthogonal and semibiorthogonal systems in $C(K_{2n})$'s}

\begin{lemma}\label{lemmanumbers}
Suppose that $\theta>\rho>0$,  $n\in\N$, $n\geq 2$, $r_1,..., r_{2n}$ are reals such that
\begin{enumerate}
\item $|\sum_{1\leq i\leq 2n}r_i|<\rho$,
\item there is $1\leq i_0\leq 2n$ such that $r_{i_0}>\theta$,
\item there is $1\leq i_1\leq 2n$ such that $r_{i_1}=0$.
\end{enumerate}
Then there are $1\leq i,j\leq 2n$ such that
$(-1)^{i+j}=-1$ and 
$$r_i+r_j< {2n\rho-\theta\over n(2n-2)}.$$
\end{lemma}

\begin{proof}
By (1) and (2), since $\theta>\rho$, there must be an $i_2\in\{1,..., 2n\}
\setminus\{i_0,i_1\}$
such that $$r_{i_2}<-{\theta-\rho\over 2n-2}={\rho-\theta\over 2n-2}
<{2n\rho-\theta\over n(2n-2)}.$$
So, if there is $i_3$ such that $(-1)^{i_2+i_3}=-1$, and
$r_{i_3}\leq 0$, then we are done.
Otherwise, there are at least $n$ positive numbers $r_i$
(at least all $r_i$'s for $i$ of the other parity than $i_2$) and so, by (3), at
most $n-1$ negative numbers $r_i$. 
Let $k\geq 0$ be such that $n+k$ is the number
of positive $r_i$'s, let $r_{i_4}$ be the smallest positive number
among $r_i$'s  and let $r_{i_5}\leq -{\theta-\rho\over 2n-2}$
be the smallest negative number (i.e., of the biggest
absolute value) among $r_i$'s.

So we have 
$$(n+k)r_{i_4}+(n-k-1)r_{i_5}\leq
\sum\{r_i: r_i>0\}+\sum\{r_i: r_i<0\}<\rho$$

So, $$n(r_{i_4}+r_{i_5})+k(r_{i_4}-r_{i_5})<\rho+r_{i_5}
\leq \rho-{\theta-\rho\over 2n-2}$$
But $ r_{i_4}-r_{i_5}$ is non-negative, so

$$r_{i_4}+r_{i_5}
<  (1/n)( \rho-{\theta-\rho\over 2n-2})=
{(2n-1)\rho-\theta\over n(2n-2)},$$
as required.
\end{proof}

\begin{lemma} Let $n\geq 2$.
Suppose that $(f_\alpha)_{\alpha<\omega_1}$ is a sequence
of continuous rational simple functions on $K_{2n}$ as in Theorem
\ref{generictheorem} and 
$(\mu_\alpha)_{\alpha<\omega_1}$ is a sequence of 
$(2n-1)$-supported atomic Radon measures on $K_{2n}$. Then either
there are $\alpha<\beta<\omega_1$ such that 
$$|\int f_\alpha d\mu_\beta|>0.01/2n^2(2n-2)\leqno a)$$
or there is $\alpha\in\omega_1$ such that 
$$\int f_\alpha d\mu_\alpha<0.99\leqno b)$$
or there are $\alpha<\beta<\omega_1$ such that 
$$\int f_\beta d\mu_\alpha<-0.89/2n^2(2n-2).\leqno c)$$
\end{lemma}

\begin{proof} By the separability of $C_0\equiv C({2}^\omega)$ (see Definition
 \ref{cantornotation}), Lemma \ref{lemmaformoffunction} 
and thinning out the sequence, we may assume that for all
$\alpha<\omega_1$ we have
$$f_\alpha=g+\sum_{1\leq i\leq k}\sum_{1\leq l \leq 2n-1}
q_{i,l}\chi_{
 A_{\xi_\alpha^i,l}\cap V_{s_i}}$$
for some simple rational function $g\in C_0$, 
$F_\alpha=\{\xi_\alpha^1,..., \xi_\alpha^k\}\subseteq\omega_1$,
some $s_i\in 2^{m_i}$, $m_i\in N$ and some rationals
$q_{i,l}$, $1\leq i\leq k$ and $1\leq l\leq 2n$
such that  $s_i=r_{\xi_\alpha^i}|m_i$ and such that
$$\sum_{1\leq i\leq k}(\max_{1\leq l\leq 2n}|q_{i,l}|)|\mu_\alpha|(V_{s_i}\setminus R_{\xi_\alpha^i})\leq 0.01/2n^2(2n-2).$$
By thinning out the sequence
(applying the $\Delta$-system lemma, see [Ku])  and moving some identical parts to
$g$ we may assume that $F_\alpha$'s are pairwise disjoint
and $g$ (no longer in $C_0$) is fixed. So, we will be allowed
to use the following decompositions:
\vskip 6pt
\noindent{\textbf{Claim 0.}}
\textit{For each $\alpha,\beta<\omega_1$ we have}
$$\int f_\alpha d\mu_\beta =\int g d\mu_\beta+
\sum_{1\leq i\leq k}
\sum_{1\leq l\leq 2n-1}q_{i,l}\mu_\beta({A_{\xi_\alpha^i,l}\cap R_{\xi_\beta^i}\cap V_{s_i}})
+$$
$$+\sum_{1\leq i\leq k}\sum_{1\leq l\leq 2n-1}q_{i,l}\mu_\beta({A_{\xi_\alpha^i,l}\cap V_{s_i} \setminus R_{\xi_\beta^i}})$$
\vskip 6pt
Here, the last term is small by the above application of Lemma \ref{lemmaformoffunction},
the first term will be shown small by the claim below and so the value of the
integral will depend on the relation of the points from  $R_{\xi_\beta^i,l}$
with the sets $A_{\xi_\alpha^i,l}$ which is ``as we wish" on any uncountable
set by Theorem \ref{generictheorem}.
\vskip 6pt 
\noindent {\textbf{Claim 1.}} \textit{Either a) holds or
for all but countably many $\alpha$'s in $\omega_1$  we have}
$$|\int gd\mu_\alpha|\leq 0.02/2n^2(2n-2).$$
\noindent \textit{Proof of the claim.} 
If the inequality does not hold for uncountably many $\alpha$s, then by
 Theorem \ref{generictheorem} we can find among them $\alpha<\beta<\omega_1$
such that $R_{\xi_\beta^i}\subseteq  A_{\xi_\alpha^i,2n}$
for all $1\leq i\leq k$. By Claim 0 we get that
$$|\int f_\alpha d\mu_\beta|\geq |\int g d\mu_\beta|-
\sum_{1\leq i\leq k}
\sum_{1\leq l\leq 2n-1}|q_{i,l}||\mu_\beta({A_{\xi_\alpha^i,l}\cap V_{s_i}})|\geq$$
$$\geq |\int g d\mu_\beta|-
\sum_{1\leq i\leq k}(\max_{1\leq l\leq m}|q_{i,l}|)|\mu_\beta|(V_{s_i} \setminus R_{\xi_\beta^i})>{0.02-0.01\over2n^2(2n-2)}=\frac{0.01}{2n^2(2n-2)}$$
obtaining a) of the lemma and concluding Claim 1.
\vskip 6pt

\noindent{\textbf{Claim 2.}}
\textit{Either a) holds or for all but countably
 many $\alpha$'s in $\omega_1$ we have
for each $1\leq l_0\leq 2n-1$}
$$|\sum_{1\leq i\leq k}q_{i,l_0}\mu_\alpha(R_{\xi^i_\alpha})|\leq 0.04/2n^2(2n-2)$$

 \noindent \textit{Proof of the claim.}
Without loss of generality
we may assume that the condition from Claim 1 holds for
all $\alpha<\beta<\omega_1$.
Fix $l_0$ as above. 
Suppose that the condition above does not hold for uncountably many $\alpha$s,
then by Theorem \ref{generictheorem} we obtain among them  $\alpha<\beta$ such that
for all $1\leq i\leq k$
$$R_{\xi_\beta^i}\subseteq  A_{\xi_\alpha^i,l_{0}}.$$
 
So by Claim 0 we have
$$|\int f_\alpha d\mu_\beta|\geq 
|\sum_{1\leq i\leq k}\sum_{1\leq l\leq 2n-1}q_{i,l}\mu_\beta({R_{\xi^i_\beta}}\cap{A_{\xi_\alpha^i,l}\cap V_{s_i}})|
-|\int g d\mu_\beta|-$$
$$-\sum_{1\leq i\leq k}(\max_{1\leq l\leq 2n-1}|q_{i,l}|)|\mu_\beta|(V_{s_i}\setminus R_{\xi_\beta^i})\geq$$
$$\geq|\sum_{1\leq i\leq k}q_{i,l_0}\mu_\beta (R_{\xi_\beta^i})|
-|\int g d\mu_\beta|-
\sum_{1\leq i\leq k}(\max_{1\leq l\leq 2n-1}|q_{i,l}|)|\mu_\beta|(V_{s_i}\setminus R_{\xi_\beta^i})>$$
$$>{0.04-0.02-0.01\over2n^2(2n-2)}=0.01/2n^2(2n-2),$$
obtaining a) and concluding Claim 2.
\vskip 6pt

\noindent{\textbf{Claim 3.}} \textit{Either a) or b) holds or there is 
$l_0\in\{1,..., 2n\}$  such that for 
uncountably many $\alpha$'s in $\omega_1$ we have}
$$\sum_{1\leq i\leq k} q_{i,l_0}\mu_\alpha(\{(x_{\xi^i_\alpha},l_0)\})>
0.96/(2n)$$
\noindent \textit{Proof of the claim.} 
Assume that a) does not hold, i.e., the conditions
from Claim 1 and Claim 2 hold for all $\alpha<\omega_1$.  Now, suppose also
that the condition from Claim 3 does not hold for any $l_0\in\{1,...,2n\}$.
By Claim 0 for $\alpha=\beta$  we have $$\int f_\alpha d\mu_\alpha\leq
\sum_{1\leq i\leq k}\sum_{1\leq l\leq 2n-1} q_{i,l}\mu_\alpha(\{(x_{\xi^i_\alpha},l)\})+$$

$$-0.02/2n^2(2n-2)-0.01/2n^2(2n-2)\leq(2n-1)0.96/(2n)- {0.03}< 0.99$$
that is we obtain b), which concludes the proof of Claim 3.
 
\vskip 6pt
To finish the proof of the lemma, we need to  assume that a) and b) fail,
i.e., the conditions of all the above claims hold, and we need to get c).
Fix $\alpha<\omega_1$,  we will apply Lemma \ref{lemmanumbers}
for 
$$r_{l,\alpha}=\sum_{1\leq i\leq k} q_{i,l_0}\mu_\alpha(\{(x_{\xi^i_\alpha},l)\})$$
and $l\in\{1,..., 2n\}$. 
By the fact that the supports of the measures $\mu_\alpha$ have
at most $2n-1$ elements, one of $r_{l,\alpha}$'s must be zero.
By Claim 3 we have that $r_{l_0,\alpha}>\theta=0.96/(2n)$
and by Claim 2 that $\sum_{1\leq l\leq 2n}r_{l,\alpha}<\rho=0.04/(2n)^2$.
So by Lemma \ref{lemmanumbers} we find
$1\leq l_{1,\alpha},l_{2,\alpha}\leq 2n$ of different parities such that 
$$\sum_{1\leq i\leq k} q_{i,l_0}\mu_\alpha(\{(x_{\xi^i_\alpha},l_{1,\alpha}),
(x_{\xi_\alpha^i},l_{2,\alpha})\})<{2n\rho-\theta\over n(2n-2)}=$$
$$={2n(0.04/(2n)^2)-0.96/(2n)\over n(2n-2)}=-{0.92\over2n^2(2n-2)}.$$
We may assume that $l_{1,\alpha}=l_1$ and $l_{2,\alpha}=l_2$ for
all $\alpha<\omega_1$.
Note that 
 by Theorem \ref{generictheorem}
we can find $\alpha<\beta<\omega_1$ 
such that 
$$\{ (x_{\xi_\alpha^i},l_1),   (x_{\xi_\alpha^i},l_2) \}\subseteq A_{\xi_\beta^i,l_0}$$
and
$$R_{\xi^i_\alpha}\setminus   \{ (x_{\xi_\alpha^i},l_1),   (x_{\xi_\alpha^i},l_2) \}   \subseteq A_{\xi_\beta^i, 2n}$$
for all $1\leq i\leq k$. 
This and Claim 0 with $\alpha$ and $\beta$ switched implies that
$$\int f_\beta d\mu_\alpha\leq \sum_{1\leq i\leq k}q_{i,l_0}\mu_\alpha(\{ (x_{\xi_\alpha^i},l_1), (x_{\xi_\alpha^i},l_2)\})+$$
$$+\sum_{1\leq i\leq k}\max_{1\leq l\leq 2n-1}(|q_{i,l}|)|\mu_\alpha|(V_{s_i}\setminus R_{\xi_\alpha^i})+ |\int g d\mu_\alpha|\leq$$ 
$$\leq {-0.92+0+0.01+0.02\over (2n)^2(2n-2)}=-0.89/(2n)^2(2n-2)$$
which completes the proof of the lemma.
\end{proof}

\begin{theorem}  Let $n\geq 2$.
If $K_{2n}$ is an unordered $2n$-split Cantor set as in Theorem \ref{generictheorem}, then 
there are no uncountable semibiorthogonal sequences 
in $C(K_{2n})$  
whose functionals are $(2n-1)$-supported but there is an uncountable biorthogonal
system whose functionals are $2n$-supported.
\end{theorem}

\begin{proof}
 Suppose $(f_\alpha,\mu_\alpha)_{\alpha<\omega_1}\subseteq C(K_{2n})\times
M(K_{2n})$ is a sequence whose functionals are $2n-1$-supported
and that $\int f_\alpha d\mu_\beta=0$ for all $\alpha<\beta<\omega_1$ as
well as $\int f_\alpha d\mu_\alpha=1$ for all $\alpha<\omega_1$.

 We may assume without loss of generality that
$||\mu_\alpha||\leq M$ for some positive $M$. By the
Stone Weierstrass theorem we can chose $f_\alpha'\in C(K)$
which is a rational simple function and
$$||f_\alpha'-f_\alpha||< 0.01/2Mn^2(2n-2).$$
This means that a) and b) of Lemma 14 do not hold, for
$f_\alpha'$'s instead of $f_\alpha$'s i.e.,
c) holds, which implies that $(f_\alpha,\mu_\alpha)_{\alpha<\omega_1}$
is not semibiorthogonal.
\end{proof}

\begin{theorem}
If $K_{2n}$ is an unordered $2n$-split Cantor set as in Theorem \ref{generictheorem}, then $hd(K_{2n}^n)=\omega$.
\end{theorem}
\begin{proof}
We will be using the well-known fact that a regular space is hereditarily separable
if and only if it has no uncountable left-separated sequence (see Theorem 3.1 of \cite{roitman}).

Suppose $(y_\alpha)_{\alpha<\omega_1}$ is a
left-separated sequence in $K_{2n}^n$ of cardinality $\aleph_1$. Hence, for each
$\alpha<\omega_1$, $y_\alpha = (y_\alpha^1, \dots,
y^n_\alpha)$, where each $y_\alpha^m \in K_{2n}$ and, by 
the definition of a left-separated sequence, for each $\alpha < \omega_1$ and each $m \in [N]$, 
there is an open basic neighbourhood $U_\alpha^m$ of $y^m_\alpha$ such that
$$\forall \alpha < \omega_1 \quad \forall m \in [n] \quad y^m_\alpha \in U^m_\alpha$$
and that
$$\forall \alpha < \beta < \omega_1 \quad \exists m \in [n] \quad y^m_\alpha \notin U^m_\beta.$$

We may assume without loss of generality that 
$$\{m \in [n]: y_\alpha^m \in 2^\omega \setminus \{x_\xi: \xi< \omega_1\}\} 
= \{m \in [n]: y_\beta^m \in 2^\omega  \setminus \{x_\xi: \xi< \omega_1\}\}$$
for every $\alpha<\beta<\omega_1$ and let us call this set $I$. 

For each $m \in [n] \setminus I$, let $\xi^m_\alpha$ be a countable ordinal and 
$j_\alpha^m$ be an element of $[n]$ such that $y^m_\alpha = (x_{\xi^m_\alpha}, j^m_\alpha)$.

Now, for each $m \in [n]$, let $s^m_\alpha \in 2^{<\omega}$ such that
$$U^m_\alpha = \left\{\begin{array}{cl}
V_{s^m_\alpha} & \text{if } m \in I\\
V_{s^m_\alpha} \cap A_{\xi^m_\alpha, j^m_\alpha}& \text{if } m \notin I
\end{array}\right.$$

Put $E_\alpha = \{\xi^m_\alpha: m \in [n] \setminus I\}$.

Without loss of generality, we may assume that:
\begin{itemize}
\item there is $j_m \in [n]$ such that $j^m_\alpha = j_m$ for all $\alpha<\omega_1$;
\item there is $s_m \in 2^{<\omega}$ such that $s^m_\alpha = s_m$ for all $\alpha<\omega_1$ (this already guarantees that
each $y^m_\alpha \in V_{s_m}$);
\item for all $m \in [n] \setminus I$, either 
$$\forall \alpha<\beta<\omega_1 \quad \xi^m_\alpha = \xi^m_\beta,$$
or 
$$\forall \alpha<\beta<\omega_1 \quad \xi^m_\alpha < \xi^m_\beta.$$
\item $(E_\alpha)_{\alpha<\omega_1}$ is a $\Delta$-system with root $\Delta$ such that for every $\alpha<\beta<\omega_1$, 
$\Delta < E_\alpha \setminus \Delta < E_\beta \setminus \Delta$ and $|E_\alpha| = |E_\beta|$.
\end{itemize}

If $E_\alpha \setminus \Delta = \emptyset$, the left-separated sequence in $K_{2n}^n$ would lead 
to a left-separated sequence in a finite power of $2^\omega$, which is not possible since 
$2^\omega$ is hereditarily separable in all finite powers. Therefore, each $E_\alpha \setminus \Delta \neq \emptyset$ and
they are pairwise disjoint.

For each $\alpha<\omega_1$, enumerate $E_\alpha \setminus \Delta = \{\eta^1_\alpha < \dots < \eta^k_\alpha\}$.
We may assume that $\xi^m_\alpha = \eta^i_\alpha$ if and only if $\xi^m_\beta = \eta^i_\beta$.

\vspace{6pt}

\noindent \textbf{Claim.} \textit{For each $1 \leq i \leq k$, 
it is possible to find $I_i \subseteq [2n]$ of cardinality $N$ and a bijection $\sigma_i: I_i \rightarrow [2n] \setminus I_i$
such that $\sigma_i(l)$ and $l$ have opposite parity and}
$$\{j \in [2n]: \exists m \in [n] \text{ such that } j = j_m \text{ and } \xi^m_\alpha = \eta^i_\alpha\} \subseteq I_i.$$

\noindent \textit{Proof of the claim.} It follows simply from the fact that the set
$$\{j \in [2n]: \exists m \in [n] \text{ such that } j = j_m \text{ and } \xi^m_\alpha = \eta^i_\alpha\}$$
has cardinality at most $n$ so that we can find $I_i$ containing it and that whenever we have a partition of $[2n]$ into two sets $A$ and $B$, both of size $n$, 
$A$ has as many odds as $B$ has evens, and vice-versa. This concludes the proof of the claim.

\vskip 6pt
Now, let $\epsilon:[k]\times[2n] \rightarrow [2n]$ be defined by
$$\epsilon (i, l) = \left\{\begin{array}{cl}
l & \text{if } l \in I_i\\
\sigma_i^{-1}(l) & \text{if } l \in [2n] \setminus I_i
\end{array}\right.$$

Notice that for each $i \in [k]$, $l \in I_i$ and $j \in [2n]$, $\epsilon(i,l) = j$ if and only if $\epsilon(i,\sigma(l)) = j$. Since $\sigma(l)$ and $l$ have
opposite parities, we get that $\epsilon$ has the desired property, that is,
$$|\{l\in\{1,3,5,...,2n-1\}: \epsilon(i,l)=j\}|= |\{l\in\{2,4,6,...,2n\}: \epsilon(i,l)=j\}|.$$

By Theorem \ref{generictheorem}, there are $\alpha< \beta$ such that for all $i \in [k]$, 
$$(x_{\eta^i_\alpha}, l) \in A_{\eta^i_\beta, \epsilon(i,l)}.$$

Fix $m \in [n]$ and let us prove that $y^m_\alpha \in U^m_\beta$, contradicting the assumption. If
$m \notin I$, then  $y_\alpha^m  \in V_{s_m} = U_\beta^m$. If $m \in I$ and $\xi_\alpha^m \in \Delta$, then
$\xi^m_\alpha = \xi^m_\beta \in U^m_\beta$. Finally, if $m \in I$ and $\xi^m_\alpha \notin \Delta$,
then there is $i \in [k]$ such that $\xi^m_\alpha = \eta^i_\alpha$ and $\xi^m_\beta = \eta^i_\beta$. In this case we
have that $j_m \in I_i$ and so, $\epsilon(i,j_m)=j_m$, which guarantees that
$$y^m_\alpha =  (x_{\xi^m_\alpha}, j_m) = (x_{\eta^i_\alpha}, j_m) \in A_{{\eta^i_\beta}, j_m} = A_{{\xi^m_\beta}, j_m}.$$
Since also $y^m_\alpha  \in V_{s_m}$, we get that $y^m_\alpha  \in U^m_\beta$, which concludes the proof.
\end{proof}

\bibliographystyle{amsplain}

\begin{thebibliography}{20}

\bibitem{bellginsburgstevo}
M.~Bell, J.~Ginsburg, and S.~Todor{\v{c}}evi{\'c}, \emph{Countable spread of
  {${\rm exp}Y$} and {$\lambda Y$}}, Topology Appl. \textbf{14} (1982), no.~1,
  1--12.

\bibitem{borwein}
J.~M. Borwein and J.~D. Vanderwerff, \emph{Banach spaces that admit support
  sets}, Proc. Amer. Math. Soc. \textbf{124} (1996), no.~3, 751--755.

\bibitem{brechkoszmider}
C.~Brech and P.~Koszmider, \emph{Thin-very tall compact scattered spaces which
  are hereditarily separable}, accepted to Trans. Amer. Math. Soc. (2009).

\bibitem{mirnaistvan} M. D\v zamonja, I. Juh\'asz; \emph{CH, a problem of Rolewicz and bidiscrete systems},  to appear in  Top. App.

\bibitem{engelking}
R.~Engelking, \emph{General topology}, second ed., Sigma Series in Pure
  Mathematics, vol.~6, Heldermann Verlag, Berlin, 1989, Translated from the
  Polish by the author.

\bibitem{fabianetal}
M.~Fabian, P.~Habala, P.~H{\'a}jek, V.~Montesinos~Santaluc{\'{\i}}a, J.~Pelant,
  and V.~Zizler, \emph{Functional analysis and infinite-dimensional geometry},
  CMS Books in Mathematics/Ouvrages de Math\'ematiques de la SMC, 8,
  Springer-Verlag, New York, 2001.

\bibitem{godefroy}
C.~Finet and G.~Godefroy, \emph{Biorthogonal systems and big quotient spaces},
  Banach space theory ({I}owa {C}ity, {IA}, 1987), Contemp. Math., vol.~85,
  Amer. Math. Soc., Providence, RI, 1989, pp.~87--110.

\bibitem{convex}
A.~S. Granero, M.~Jim{\'e}nez~Sevilla, and J.~P. Moreno, \emph{Convex sets in
  {B}anach spaces and a problem of {R}olewicz}, Studia Math. \textbf{129}
  (1998), no.~1, 19--29.

\bibitem{biorthogonal}
P.~H{\'a}jek, V.~Montesinos~Santaluc{\'{\i}}a, J.~Vanderwerff, and V.~Zizler,
  \emph{Biorthogonal systems in {B}anach spaces}, CMS Books in
  Mathematics/Ouvrages de Math\'ematiques de la SMC, 26, Springer, New York,
  2008.

\bibitem{katetov} 
M.~Kat{\v{e}}tov, \emph{Complete normality of {C}artesian products}, Fund.
  Math. \textbf{35} (1948), 271--274.

\bibitem{piotrrolewicz}
P.~Koszmider, \emph{On a problem of {R}olewicz about {B}anach spaces that admit
  support sets}, J. Funct. Anal. \textbf{257} (2009), no.~9, 2723--2741.

\bibitem{kunen}
K.~Kunen, \emph{Set theory. an introduction to independence proofs}, Studies in
  Logic and the Foundations of Mathematics, vol. 102, North-Holland Publishing
  Co., Amsterdam, 1980.

\bibitem{lazar}
A.~J. Lazar, \emph{Points of support for closed convex sets}, Illinois J. Math.
  \textbf{25} (1981), no.~2, 302--305.

\bibitem{monk}
J.~D. Monk, \emph{Cardinal functions on {B}oolean algebras}, Lectures in
  Mathematics ETH Z\"urich, Birkh\"auser Verlag, Basel, 1990.

\bibitem{negrepontis}
S.~Negrepontis, \emph{The {S}tone space of the saturated {B}oolean algebras},
  Trans. Amer. Math. Soc. \textbf{141} (1969), 515--527.

\bibitem{roitman}
J.~Roitman, \emph{Basic {$S$} and {$L$}}, Handbook of set-theoretic topology,
  North-Holland, Amsterdam, 1984, pp.~295--326.

\bibitem{rolewicz}
S.~Rolewicz, \emph{On convex sets containing only points of support}, Comment.
  Math. Special Issue \textbf{1} (1978), 279--281, Special issue dedicated to
  W\l adys\l aw Orlicz on the occasion of his seventy-fifth birthday.

\bibitem{shelah1}
S.~Shelah, \emph{On uncountable {B}oolean algebras with no uncountable pairwise
  comparable or incomparable sets of elements}, Notre Dame J. Formal Logic
  \textbf{22} (1981), no.~4, 301--308.

\bibitem{shelah2}
\bysame, \emph{Uncountable constructions for {B}.{A}., e.c.\ groups and
  {B}anach spaces}, Israel J. Math. \textbf{51} (1985), no.~4, 273--297.

\bibitem{singer1}
I.~Singer, \emph{Bases in {B}anach spaces. {I}}, Springer-Verlag, New York,
  1970, Die Grundlehren der mathematischen Wissenschaften, Band 154.

\bibitem{singer2}
\bysame, \emph{Bases in {B}anach spaces. {II}}, Editura Academiei Republicii
  Socialiste Rom\^ania, Bucharest, 1981.

\bibitem{sspaces} 
Z.~Szentmikl{\'o}ssy, \emph{{$S$}-spaces and {$L$}-spaces under {M}artin's
  axiom}, Topology, {V}ol. {II} ({P}roc. {F}ourth {C}olloq., {B}udapest, 1978),
  Colloq. Math. Soc. J\'anos Bolyai, vol.~23, North-Holland, Amsterdam, 1980,
  pp.~1139--1145.

\bibitem{stevoirr}
S.~Todorcevic, \emph{Irredundant sets in {B}oolean algebras}, Trans. Amer.
  Math. Soc. \textbf{339} (1993), no.~1, 35--44.

\bibitem{stevobio}
\bysame, \emph{Biorthogonal systems and quotient spaces via {B}aire category
  methods}, Math. Ann. \textbf{335} (2006), no.~3, 687--715.

\end{thebibliography}

\end{document}